\documentclass[reqno]{amsart}
\usepackage{preamble}

\title{Permutations from ranking independent random variables}
\author{Luke Turvey}
\thanks{Funded in part by the Royal Society through a Wolfson
Fellowship held by Alexander Holroyd.}
\date{\today}

\address{School of Mathematics, University of Bristol BS8 1UG, U.K}
\email{jf19820@bristol.ac.uk}


\begin{document}
\begin{abstract} 
    Let $X_1,\ldots,X_n$ be almost surely distinct independent real-valued random variables. Let $\sigma$ be the random permutation of $\{1,\ldots,n\}$ such that $X_{\sigma(1)}<X_{\sigma(2)}<\cdots< X_{\sigma(n)}$. We show that the set of laws of $\sigma$, as the laws of $X_1,\ldots,X_n$ vary, has semialgebraic dimension
    \[
    \sum_{k=2}^n {n \choose k}(k-1)!
    \]
    as a subset of the $(n!-1)$-simplex.
    
    This establishes a conjecture of Babson, Duchin, Iseli, Poggi-Corradini, Thurston, and Tucker-Foltz who proved that the dimension is upper bounded by the above expression.  We prove their conjecture by constructing a family of finitely supported laws for $X_1,\ldots,X_n$ that provides the correct dimension. We also give an alternative proof of the upper bound using a theorem of Radford. Furthermore, we show that the Mallows law on permutations cannot arise as a law of $\sigma$ for $n\geq 4$, answering a question of Ed Crane.
\end{abstract}
\maketitle

\section{Introduction}
Let $X_1,\ldots,X_n$ be independent $\R$-valued random variables with laws $\mu_1,\ldots,\mu_n$. We make the additional assumption that $\P(X_i = X_j)=0$ for $i \neq j$. Equivalently $\mu_i$ and $\mu_j$ have no common atoms for $i \neq j$. 
Let $S_n$ denote the symmetric group on $n$ elements. For each $\sigma \in S_n$ define
\begin{equation} \label{eq:p_sigma_def}
p_\sigma := \P(X_{\sigma(1)}<X_{\sigma(2)}<\cdots< X_{\sigma(n)}).
\end{equation}
Thus $(p_\sigma)_{\sigma \in S_n}$ is a probability distribution on $S_n$. We call this law \textbf{the ranking law of $\mu_1,\ldots, \mu_n$}. We say a probability distribution $p$ on $S_n$ is \textbf{achievable} if some probability measures $\mu_1,\ldots,\mu_n$ have ranking law $p$.
Let $\Delta(S_n)$ denote the simplex of probability distributions on $S_n$, namely the set of non-negative vectors $(p_\sigma)_{\sigma \in S_n}$ whose elements sum to $1$. Let
\[
\AA_n := \{p \in \Delta(S_n) : p \text{ is achievable} \}
\]
be the set of achievable laws. Babson, Duchin, Iseli, Poggi-Corradini, Thurston  and Tucker-Foltz \cite{babson} proved that $\AA_n$ is semialgebraic, meaning that it can be described by a finite Boolean combination of polynomial equalities and inequalities in the variables $p_\sigma$.

For a semialgebraic set $S \subset \R^d$, we write $\dim S$ for its real semialgebraic dimension. For the definition, see Appendix \ref{app:inverse_func_thm} or Chapter $2$ in Bochnak, Coste and Roy \cite{bochnak_coste_roy_real_algebraic_geometry}. For the reader's convenience we collect relevant results about semialgebraic dimension in Appendix \ref{app:inverse_func_thm}. The simplex $\Delta(S_n)$ has semialgebraic dimension $n!-1$. 
\begin{theorem} \label{thm:dimension}
    We have 
    \[
    \dim(\AA_n) = \sum_{k=2}^n {n \choose k}(k-1)!
    \]
    for all $n \geq 2$.
\end{theorem}
For notational simplicity, let $d(n) = \sum_{k=2}^n {n \choose k}(k-1)!$. As $n \rightarrow \infty$, $d(n) \sim e (n-1)!$ so the dimension of $\AA_n$ is asymptotically a fraction $e/n$  of the ambient dimension of $\Delta(S_n)$. It has been shown in \cite{babson} that $\dim \AA_n \leq d(n)$. We give a new proof of this upper bound in Section \ref{sec:upper_bound} using a result from Radford \cite{radford}. The lower bound $\dim \AA_n \geq d(n)$ was further conjectured in \cite{babson}, see Conjecture $5.14$, and experimentally verified for $n \leq 7$. We prove this conjecture in Section \ref{sec:lower_bound}.

Sorting laws of independent random variables have been studied in random utility theory. The exact model we consider has previously been studied by Joe \cite{Joe2000}, in which a collection of algebraic inequalities of the probabilities $(p_\sigma)_{\sigma \in S_n}$ was derived. These inequalities were used as tests to check whether ranking data came from an independent random utility model.

More recently, the authors of \cite{babson} studied the same model in connection with random minimum spanning trees. Here, independent edge weights on the complete graph induce a ranking law on the set of edges. This ranking law is pushed forward to a law on minimum spanning trees by Kruskal's algorithm. As mentioned above, they showed that the set $\AA_n$ is semialgebraic and established the upper bound in Theorem \ref{thm:dimension}.

A special case of these ranking laws arises in the study of go first dice \cite{ford2023gofirst,purcell2024building}. Go first dice are a set of finitely supported measures (which can be thought of as dice) whose ranking law is the uniform law on $S_n$. The authors of \cite{babson} show that every ranking law can be realised by finitely supported measures, see Theorem $5.4$ \cite{babson}. A central question is how small the supports of these measures can be. We show that a full dimensional subset of $\AA_n$ can be realised by finitely supported measures, each with supports whose size is upper bounded by $\lfloor e(n-1)! \rfloor$, see Remark \ref{rem:support_bound}.

 Now we summarise the main ideas used in the proof of Theorem \ref{thm:dimension}. As noted in \cite{babson}, a key observation is that orders on disjoint sets of indices are independent. More precisely for disjoint sets of indices $A=\{a_1,\ldots,a_k\}$ and $B=\{b_1,\ldots,b_\ell\}$ we have
\begin{equation} \label{eq:disjoint_indices}
\P(X_{a_1}<\cdots<X_{a_k} \text{ and } X_{b_1}<\cdots<X_{b_\ell})=\P(X_{a_1}<\cdots<X_{a_k})\P(X_{b_1}<\cdots<X_{b_\ell}).
\end{equation}

In fact this defines a semialgebraic set that every ranking law must lie in. Using a result of Radford \cite{radford} this semialgebraic set turns out to have dimension $d(n)$. This provides the upper bound. This more streamlined proof differs from the one given in \cite{babson} which constructs sufficiently many independent equations satisfied by all ranking laws and shows their derivatives at the uniform law are independent.

For the lower bound, we construct a family of measures $\mu_1(\theta),\ldots,\mu_n(\theta)$ smoothly parameterised by $\theta$ taking values in some parameter space $\Theta\subset \R^{d(n)}$. Let $F:\Theta \rightarrow \Delta(S_n)$ be the map $F:\Theta \rightarrow \Delta(S_n)$ defined by
\[
F(\theta) =  (p_\sigma(\theta))_{\sigma \in S_n}
\]
where $p_\sigma(\theta)$ is given by \eqref{eq:p_sigma_def} for independent $X_i \sim \mu_i(\theta), 1 \leq i \leq n$. A standard consequence of the inverse function theorem, see Proposition \ref{prop:inverse_func_thm}, is that if $F:U \subset \R^d \rightarrow \Delta(S_n)$ is a smooth map with $F(U) \subset S$ and $\text{rank}( D F(\theta_0)) \geq d(n)$ for some $\theta_0 \in U$ then $\dim S \geq d(n)$. Here $D$ denotes the differential operator. 

The set of measures $\mu(\theta)=(\mu_1(\theta),\ldots, \mu_n(\theta))$ we construct are in fact finitely supported atomic measures where $\theta$ is a vector of atom masses. 

We further show that algebraic relations such as \eqref{eq:disjoint_indices} exclude some very natural probability laws on the symmetric group. In particular, the Mallows law is not achievable.
For a permutation $\sigma \in S_n$ define
\[
\text{inv}(\sigma) = \#\{i<j: \sigma(i)>\sigma(j)\}.
\]

For $q \geq 0$, the Mallows$_n(q)$ law on $S_n$ is defined by 
\[
\P_q(\sigma) = \frac{q^{\text{inv}(\sigma)}}{Z_n(q)}
\]

where $Z_n(q):= \sum_{\sigma \in S_n} q^{\text{inv}(\sigma)}$.

\begin{theorem} \label{thm:mallows_unachievable}
    For $q \neq 0,1$ and $n \geq 4$ the Mallows$_n(q)$ law on $S_n$ is not achievable.
\end{theorem}

\section*{Acknowledgement}
I thank Marton Balazs, Alexander Holroyd and Daniel Meriaz for valuable discussions and suggestions that improved this work. I am particularly grateful to Ed Crane for posing the question of whether the Mallows law is achievable and for valuable discussions concerning this problem. I also acknowledge the use of ChatGPT (OpenAI) for copy-editing and suggestions concerning the exposition. All mathematical claims are my responsibility.
\section{Preliminaries}

For any measurable space $(\Omega,\FF)$ define $\Delta(\Omega)$ as the set of probability measures on that space. Throughout we assume $S_n$ and $\R$ are equipped with the discrete and Borel $\sigma$-algebras, respectively. We say the measures $(\mu_1,\ldots,\mu_n) \in \Delta(\R)^n$ are \textbf{non-colliding} if for independent random variables $X_1,\ldots,X_n$ with laws $\mu_1,\ldots,\mu_n$ we have
\[
\P(X_i=X_j)=0 \quad \text{ for } i \neq j.
\]
Let $\Xi_n$ be the space of non-colliding $n$-tuples of probability measures.
We define the \textbf{ranking law map} $F:\Xi_n \rightarrow \Delta(S_n)$ as $F(\mu_1,\ldots,\mu_n) = (p_\sigma)_{\sigma \in S_n}$ where
\[
p_\sigma := \P(X_{\sigma(1)}<X_{\sigma(2)} < \cdots < X_{\sigma(n)})
\]
for independent random variables $X_1,\ldots,X_n$ with laws $\mu_1,\ldots,\mu_n$. We say a law $p \in \Delta(S_n)$ is \textbf{achievable} if there exists some $\mu \in \Xi_n$ such that $F(\mu)=p$. Denote the set of achievable laws in $\Delta(S_n)$ by $\AA_n$.

Let $[n] = \{1,\ldots,n\}$. Throughout, a \textbf{word} on $[n]$ means a finite sequence $w=w_1 w_2 \ldots w_k$ of elements of $[n]$. We write $\text{supp}(w)=\{w_1,\ldots, w_k\}$ as the \text{support} of $w$.  We call a word containing no repeated letters \textbf{injective}. Let $A \subset [n]$ and write $\sigma |_A$ as the subword obtained by deleting all letters not in $A$ and preserving the order of the remaining letters. For two words $u,v$ on $[n]$ we write $u<v$ to denote the \textbf{lexicographic order}. Namely $u<v$ if $u$ is a proper prefix of $v$ or when the letter at which they differ is smaller in $u$ than in $v$. A non-trivial rotation of a word $w$ is a word $vu$ arising from a factorisation $w=uv$ where $u$ and $v$ are \textbf{non empty}.

We identify a permutation $\sigma \in S_n$ with the word $\sigma(1)\sigma(2) \cdots \sigma(n)$. Let $p=(p_\sigma)_{\sigma \in S_n} \in \Delta(S_n)$. For any injective word $w$ on $[n]$ define 
\[
q_w(p) = \sum_{\substack{\sigma \in S_n :\\ 
 \sigma |_{\supp(w)} =w}} p_\sigma.
\]
Note that if $w$ has support $[n]$ then $q_w(p) = p_w$.

Suppose $u$ and $v$ are words with disjoint supports. A \textbf{shuffle} of $u$ and $v$ is a word $w$ such that $\supp(w)= \supp(u) \cup \supp(v)$ and $w|_{\supp (u)}=u$ and $w|_{\supp (v)} = v$. We denote the set of all shuffles $u \circ v$.

For example, if $u=13, v=24$ then $u \circ v= \{1234,1324,1243,2134,2143, 2413\}$.

For a set of words $A$ each with support disjoint from $u$ define $A \circ u = u \circ A= \cup_{v \in A} u \circ v$. For words $u,v,w$ with pairwise disjoint supports write $u \circ v \circ w= u \circ(v \circ w)$ and note that $\circ$ is associative and commutative.

If $p \in \Delta(S_n)$ is an achievable law, then for every pair of injective words $u,v$ with disjoint supports, we have
\begin{equation} \label{eq:disjoint_block_eq}
q_u(p) q_v(p) = \sum_{w \in u \circ v} q_w(p).
\end{equation}
This follows from the observation that if $p$ is achievable then it can be realised by independent random variables $X_1,\ldots,X_n \sim \mu_1,\ldots,\mu_n$ where orders on disjoint indices are independent.

\section{Upper bound} \label{sec:upper_bound}
In this section we give a new proof that $\dim \AA_n \leq d(n)$ already proved in \cite{babson}. Denote $V_n \subset \Delta(S_n)$ as the set of laws that satisfy the shuffle identities \eqref{eq:disjoint_block_eq}. We know $\AA_n \subset V_n$ and hence $\dim(\AA_n) \leq \dim(V_n)$.

The following results are a consequence of a theorem of Radford \cite{radford}. We provide an exposition of the necessary results for completeness following the paper by Melançon and Reutenauer \cite{melancon}. Their result holds in the more general context of words with repeated letters. We only provide the results for injective words.

\begin{definition}
    A \textbf{Lyndon word} is a nonempty word that is strictly smaller in lexicographic order than all of its non-trivial rotations.
\end{definition}

For example, $132$ is a Lyndon word but $321$ and $213$ are not since \[
132<213<321.
\]

\begin{theorem} [Lyndon, Theorem $5.1.5$ in \cite{lothaire}] \label{thm:lyndon_factorisation}
    Any word $w$ may be written uniquely as a nonincreasing product of Lyndon words. More precisely any word $w$ of $[n]$ can be written as
    \[
    w = \ell_1\ell_2\cdots \ell_s, \qquad \ell_1\geq\ell_2\geq\cdots \geq\ell_s
    \]
    for Lyndon words $\ell_1,\ell_2,\ldots,\ell_s$.
\end{theorem}
The following is the key result that enables us to prove the upper bound. 
\begin{theorem}[Radford \cite{radford}] \label{thm:radford}
    Let $w$ be a word with distinct letters and let 
    \[
    w = \ell_1 \ell_2 \cdots \ell_s
    \] 
    be its Lyndon factorisation. Then $w$ is the lexicographically largest word appearing in 
    \[
    \ell_1 \circ \ell_2 \circ \cdots \circ \ell_s.
    \] 
\end{theorem}
For completeness, we include the proof of Theorem \ref{thm:radford} in Appendix \ref{app:radford}. We follow the proof from Section $4$ in \cite{melancon} which cites the original result from Radford \cite{radford}.

\begin{proof} [Proof of Theorem \ref{thm:dimension} upper bound.]
 Let $w \in S_n$, which we identify with the word $w=w_1,\ldots w_n$. Let $\ell_1,\ell_2,\ldots \ell_s$ be its Lyndon factorisation. Hence Theorem \ref{thm:radford} and induction on \eqref{eq:disjoint_block_eq} gives us that for any $p \in V_n$,
\[
q_{\ell_1}(p)q_{\ell_2}(p)\cdots q_{\ell_s}(p) = q_w(p)+ \sum_{\substack{v \in \ell_1 \circ \ell_2 \circ \cdots \circ \ell_s: \\ v<w}} q_v(p).
\]
Rearranging, we have the following decomposition of $q_w(p)$,
\begin{equation} \label{eq:lyndon_decomp}
 q_w(p) = q_{\ell_1}(p)q_{\ell_2}(p)\cdots q_{\ell_s}(p)-\sum_{\substack{v \in \ell_1 \circ \ell_2 \circ \cdots \circ \ell_s: \\ v<w}} q_v(p).
\end{equation}
For each term in the sum, if $v$ is not a Lyndon word we decompose $q_v(p)$ using \eqref{eq:lyndon_decomp}. Iterating, we can express $q_w(p)$ as a polynomial in the variables $q_\ell(p)$ where $\ell$ is a Lyndon word. Each application of \eqref{eq:lyndon_decomp} replaces a term $q_v(p)$ by a polynmial involving terms with $q_z(p)$ with $z<v$. Since there are only finitely many injective words with the same support, this procedure terminates.

Further note that $q_\ell(p) = 1$ if $|\ell|=1$. Therefore fixing the values of $q_\ell(p)$ for each Lyndon word $\ell$ such that $|\ell| \geq 2$ along with the condition that $\sum_{\sigma \in S_n} p_\sigma =1$ completely determines $q_w(p)$ for all injective words $w$ and hence completely determines $p$. More precisely, the map 
\[
p \mapsto (q_\ell(p) : \ell \text{ is a Lyndon word and } |\ell| \geq 2)
\] 
is injective on $V_n$. The set $V_n$ is a real semialgebraic subset of $\Delta(S_n)$, and the above coordinate projection is a semialgebraic injective map from $V_n$ into $\R^{d(n)}$. Since semialgebraic dimension is monotone under injective semialgebraic maps, see \cite[Theorem 2.8.8]{bochnak_coste_roy_real_algebraic_geometry},
\[
\dim(V_n) \leq \# \{\text{injective Lyndon words of $[n]$ of length $\geq 2$}\}
\]
To conclude the proof we count the number of Lyndon words of length $\geq 2$ with distinct letters taking letters in $[n]$.

There are ${n \choose k} k!$ injective words of length $k$ and note there are $k$ distinct rotations of an injective word. Among the $k$ rotations of a word of length $k$ there is exactly one Lyndon word. Thus there are ${n \choose k} k!/k$ Lyndon words in the injective words of length $k$. Therefore
\[
\# \{\text{injective Lyndon words of $[n]$ of length $\geq 2$}\}=d(n)
\]
and $\dim \AA_n \leq \dim V_n \leq d(n)$.
\end{proof}

\section{Lower bound} \label{sec:lower_bound}
In this section we prove that $\dim \AA_n \geq d(n)$. Let $W(m)$ denote the set of non-empty injective words on $[m]$. 
Let $D(n):=|W(n-1)|$  and note that 
\begin{equation}
    d(n)-d(n-1) = \sum_{j=1}^{n-1} {n-1 \choose j} j! = D(n).
\end{equation}

For real-valued independent random variables $X_1,\ldots,X_n$ with laws $\mu_1,\ldots,\mu_n$ and an injective word $u =u_1\cdots u_k \in W(m)$ define
\[
Q_u(t) := \P(X_{u_1} < \cdots <X_{u_k}<t)
\]
for all $t \in \R$. Set $Q_\emptyset(t):=1$. 
Equivalently, 
\[
Q_u(t) := \int_{\R^k} \one\{x_1<\cdots < x_k <t\} d \mu_{u_1}(x_1)\cdots d \mu_{u_k}(x_k).
\]
Note this definition extends to finite measures as well as probability measures. 
 We say the measures $\mu$ are \textbf{good} if the functions
\[
1, \quad (Q_u)_{u \in W(m)}
\]
are linearly independent over $\R$.

We first describe a sketch of the proof of the lower bound in Theorem \ref{thm:dimension}. The proof proceeds by induction on $n$; see Theorem \ref{thm:induction}. Suppose that there exists a family of measures $\mu(\theta)=(\mu_1(\theta),\ldots,\mu_{n-1}(\theta))$ and some open set $U \subset \R^{d(n-1)}$ for which the ranking map $F(\mu(\theta)):U \rightarrow \Delta(S_{n-1})$  satisfies
\[
\rank D_\theta F(\mu(\theta_0)) \geq d(n-1)
\]
for some $\theta_0 \in U$. For each $\theta\in U$ we add finitely many atoms into the measures $\mu(\theta)$ to make them good whilst preserving this rank condition; see Lemmas \ref{lem:making_good} and \ref{lem:staying_good}. Once $\mu_{n-1}(\theta)$ is good, we construct a new measure $\mu_n(\beta)$ so that 
\[
\rank D_{\theta,\beta} F(\mu_1(\theta_0),\ldots,\mu_{n-1}(\theta_0),\mu_{n}(\beta^0)) \geq d(n)
\]
for some open set $V \subset \R^{D(n)}$ and $\beta^0 \in V$; see Lemma \ref{lem:inserting_Xn}. Proposition \ref{prop:inverse_func_thm} then implies Theorem \ref{thm:induction}.

\begin{lemma} \label{lem:lin_alg}
    Let $r \in \N$. Suppose $f_0,\ldots, f_r$ are linearly independent $\R$-valued functions. Then there exist real numbers $y_0,\ldots y_r$ such that
    \[
    \det ((f_i(y_j)))_{0\leq i,j\leq r} \neq 0.
    \]
\end{lemma}
\begin{proof}
    Suppose no such points exist. Then the span of all vectors $(f_0(y),\ldots f_r(y))$ for every $y$ has dimension $<r+1$. Thus there would be non-zero real coefficients $a_0,\ldots a_r$ such that
    \[
    a_0 f_0(y)+a_1 f_1(y)+\cdots+a_r f_r(y) = 0
    \]
    for every $y$, contradicting linear independence.
\end{proof}

\begin{lemma} \label{lem:inserting_Xn}
    Let $D=D(n)$ and $\mu=(\mu_1,\ldots,\mu_{n-1})$ be a good tuple of finite atomic probability measures with disjoint supports. There exist points $y_0,\ldots,y_{D}\in \R$ such that, for 
    \[
    V:= \left\{(\beta_1,\ldots,\beta_D)\in (0,\infty)^D : \sum_{j=1}^D \beta_j < 1\right\}
    \]
    and $\beta_0 = 1-\sum_{j=1}^D\beta_j$ the probability measure
    \[
    \mu_n(\beta) = \sum_{j=0}^{D} \beta_j \delta_{y_j} \qquad \beta \in V
    \]
     establishes the following: Let $\mu(\beta)= (\mu_1,\ldots,\mu_n(\beta))$. We have
    \[
        \rank \bra{\frac{\partial }{\partial \beta_j}q_{un}(F(\mu(\beta)))}_{u \in W(n-1), 1\leq j \leq D} = D.
    \]
    Moreover, $y_0,\ldots,y_{D(n)}$ are disjoint from $\cup_{i=1}^{n-1} \supp (\mu_i)$.
\end{lemma}

\begin{proof}
    We apply Lemma \ref{lem:lin_alg} to the functions 
    \[
    f_0 = 1, \quad f_u = Q_u, \quad u \in W(n-1)
    \]
    This gives points $y_0,\ldots, y_D$ such that
    \[
    \det \begin{pmatrix} 
            1 & Q_{u_1}(y_0) & Q_{u_2}(y_0) & \cdots & Q_{u_D}(y_0) \\
            1 & Q_{u_1}(y_1) & Q_{u_2}(y_1) & \cdots & Q_{u_D}(y_1) \\
            \vdots & \vdots & \vdots & \ddots & \vdots \\
            1 & Q_{u_1}(y_D) & Q_{u_2}(y_D) & \cdots & Q_{u_D}(y_D) 
        \end{pmatrix} 
        \neq 0
    \]
    Subtracting the first row from every other row we have that the matrix
    \begin{equation} \label{eq:jacobian_is_full_rank}
    (Q_u(y_j) - Q_u(y_0))_{1\leq j \leq D, u \in W(n-1)} \quad \text{has full rank $D$.}
    \end{equation}
    Let $X_n$ be a random variable with law 
    \[
    \mu_n(\beta) = \sum_{j=0}^{D(n)} \beta_j \delta_{y_j}
    \]
    for $\beta \in V$. Let $\mu(\beta)= (\mu_1,\ldots,\mu_n(\beta))$.
    Then, for each $u \in W(n-1)$, we have
    \[
    q_{un}(F(\mu(\beta))) = \E Q_u(X_n) = \sum_{j=0}^D \beta_j Q_u(y_j) = \bra{1- \sum_{k=1}^D \beta_k} Q_u(y_0) + \sum_{j=1}^D \beta_j Q_u(y_j)
    \]
    and hence
    \[
    \frac{\partial}{\partial \beta_j}q_{un}(F(\mu(\beta))) = Q_u(y_j)-Q_u(y_0).
    \]
    This matrix has full rank $D$ by \eqref{eq:jacobian_is_full_rank}.

    Finally note that $Q_u(t)$ is a step function whose discontinuities lie in the set $S=\cup_{i=1}^{n-1} \supp(\mu_i)$. Thus every $Q_u$ is constant on sufficiently small left-neighbourhoods of $y_i$. Therefore, we may move each $y_i$ slightly to left, if necessary, outside $S$ without changing the values of $Q_u(y_i)$ to avoid colliding atoms. Consequently, the determinant remains non-zero and $\{y_0,\ldots,y_D\} \cap S = \emptyset$.
\end{proof}

\begin{lemma} \label{lem:making_good}
    Given any vector $\mu=(\mu_1,\ldots,\mu_n)$ of probability measures with finite supports, one can add finitely many arbitrarily small atoms to each measure, so that the resulting measures are good.
\end{lemma}
\begin{proof}
It is convenient to ignore normalisation in the proof. At the end, normalising each measure multiplies each function $Q_u$ by a non-zero constant depending on $u$ which does not affect linear independence.

List $W(n)=v_1,\ldots, v_N$ in non-decreasing order of length. 

Write 
\[
    v_j = a_{j1},\ldots a_{j r_j}
\] for each $1\leq j\leq N$.
We now assign atoms to $\mu$ in the following way. Choose points in increasing order 
\[
y_{11}< \cdots y_{1 r_1}<y_{21}<\cdots<y_{2r_2}< \cdots< y_{N1}<\cdots y_{Nr_N}.
\]
larger than the supports of $\mu_1,\ldots,\mu_n$. For each $(j,k)$ add an atom of mass $\lambda_{jk}$ to $\mu_\ell$ at $y_{jk}$ if $a_{jk} = \ell$. More precisely, define a new finite measure $\mu^\lambda = (\mu_1^\lambda,\ldots \mu_n^\lambda)$ by
\[
\mu_\ell^\lambda = \mu_\ell + \sum_{(j,k)} { \lambda_{jk} \delta_{y_{j,k}}\one\{a_{jk}=\ell\}}
\]

We claim that $\mu^\lambda$ is good for arbitrarily small choices of $\lambda$. Let $B_j$ be the block of atoms corresponding to word $v_j$. For $j <N$, let $t_j$ be a point strictly between the last point of $B_j$ and the first point of $B_{j+1}$. Let $t_N$ be a point to the right of the last point of $B_N$. Now consider the $N \times N$ matrix $C$ with entries
\[
 C_{j \ell} = Q_{v_\ell}(t_j) \qquad 1 \leq j,\ell \leq N.
\]
We show that each $Q_{v_j}$ is independent by showing that $\det(C)\neq 0$. It is enough to show that the monomial
\[
\Lambda = \prod_{j=1}^N \prod_{s=1}^{r_j} \lambda_{js}
\]
has non-zero coefficient in the expansion of $\det(C)$. 

 We use the Leibniz expansion of the determinant, namely

\begin{equation} \label{eq:leibniz_expansion}
    \det(C) = \sum_{\tau \in S_N} \text{sign}(\tau) \prod_{j=1}^N C_{j,\tau(j)}.
\end{equation}

The entry $C_{jj}$ contains as an expression the product of the atoms in block $B_j$ which has mass  $\prod_{s=1}^{r_j} \lambda_{js}$. Therefore the diagonal term $C_{11}C_{22}\cdots C_{NN}$ in \eqref{eq:leibniz_expansion} contains $\Lambda$ with coefficient $+1$.

\begin{lemma} \label{lem:tau_is_id}
    Let $\tau \in S_N$ be any permutation. Suppose 
    \[
        \prod_{j=1}^N C_{j,\tau(j)}
    \]
    contains the monomial $\Lambda$. Then $\tau$ must be the identity.
\end{lemma}

Using Lemma \ref{lem:tau_is_id}, the only term in \eqref{eq:leibniz_expansion} that contains the monomial $\Lambda$ is the identity term. 
Furthermore, $\det C$ is a non-zero polynomial in the added masses $\lambda_{js}$. Therefore, its zero set has empty interior. Hence every neighbourhood of $0$ contains choices of positive $\lambda_{js}$ such that $\det(C) \neq 0$. Thus we can choose the masses $\lambda_{js}$ arbitrarily small and ensure the functions $Q_{v_1}, \ldots, Q_{v_N}$ are linearly independent.

We only need to show $1$ is linearly independent of $(Q_{v_j})_{1 \leq j \leq N}$. To see this choose a point $t_0$ to the left of all old and new support points of $\mu$. We have
\[
Q_{v_j}(t_0) = 0 \qquad \text{for all } 1\leq j \leq N.
\]
and $Q_\emptyset(t_0)=1$. Thus $Q_\emptyset$ cannot be written as a linear combination of the other functions. Therefore, the functions
\[
 1, \quad (Q_{v_j})_{1 \leq j \leq N}
\]
are linearly independent and the finite measures $\mu^\lambda=(\mu^\lambda_1,\ldots, \mu^\lambda_n)$ are good.
\end{proof}

\begin{proof}[Proof of Lemma \ref{lem:tau_is_id}]
    Consider the final block $B_N$. Every atom in $B_N$ lies after points $t_1,\ldots t_{N-1}$. Therefore, no entry in rows $1,\ldots, N-1$ of $C$ contains any mass from $B_N$. Moreover, $\Lambda$ contains every mass in $B_N$, hence each mass in $B_N$ can occur only in the final row of $C$. In particular, the entry
    \[
    C_{N,\tau(N)} := Q_{v_{\tau(N)}}(t_N)
    \]
    contains a monomial involving every mass in $B_N$.

    Viewed as a polynomial in the added masses, every monomial in $Q_{v_{\tau(N)}}(t_N)$ has degree at most $|v_N|$. Since the relevant monomial contains all $|v_N|$ masses from $B_N$ we must have 
    \[
    |v_{\tau(N)}| \geq |v_N|.
    \]
     The block $B_N$ contains $|v_N|$ many atoms and since we listed the words $(v_j)_{1\leq j \leq N}$ in length increasing order we have
    \[
    |v_{\tau(N)}| \leq |v_N|.
    \]
    However $Q_{v_{\tau(N)}}(t_N)$ must contain a term using all atoms in $B_N$ so 
    \[
    |v_{\tau(N)}| = |v_N|.
    \]
    The only word of length $|v_{\tau(N)}|$ that uses all the atoms of $B_N$ in the right order must be $v_N$. Hence
    \[
    v_{\tau(N)}=v_N,
    \]
    and thus $\tau(N)=N$.

    Now remove row $N$, column $N$ and block $N$ and apply the same argument for $N-1$. Thus $\tau(N-1)=N-1$. Continuing iteratively backwards we see
    \[
    \tau(j) = j \qquad \text{for all } 1\leq j \leq N.
    \] 
    \qedhere
\end{proof}

\begin{lemma} \label{lem:staying_good}
    Let $s \in \N$ and $\Theta \subset \R^s$ be open. Suppose there exists fixed pairwise disjoint finite sets $S_1,\ldots,S_n \subset \R$, and probability measures $\mu(\theta)=(\mu_1(\theta),\ldots,\mu_n(\theta))$ with
    \[
    \mu_i(\theta) = \sum_{x \in S_i} w_{i,x}(\theta) \delta_x, \qquad 1\leq i \leq n, \quad \theta \in \Theta,
    \]
    whose positive atom masses $w_{i,x}(\theta)$ depend smoothly on $\theta$. Further assume that for some $\theta_0 \in \Theta$,
    \[
     \rank D_\theta (F(\mu(\theta_0))) \geq r.
    \]
    Let $M \in \N$ and let $\lambda = (\lambda_1,\ldots,\lambda_M)$. Consider the new measures $\mu^\lambda=(\mu_1^\lambda,\ldots,\mu_n^\lambda)$ obtained by adding $M$ distinct Dirac masses of masses $(\lambda_j)_{1 \leq j \leq M}$ outside the supports of $\mu$ and normalising. Then 
    \[
    \rank D_\theta F(\mu^\lambda(\theta_0)) \geq r
    \]
    for sufficiently small $\lambda$.
\end{lemma}
\begin{proof}
   Since $\rank D_\theta F(\mu(\theta_0)) \geq r$ there is an $r \times r$ minor of $D_\theta F(\mu(\theta_0))$ with non-zero determinant.  Denote this determinant $\Delta (0) \neq 0$. Let $\Delta(\lambda)$ denote the determinant of the same $r \times r$ minor of $D_\theta F(\mu^\lambda(\theta_0))$.

   First denote $p_\sigma(\theta,\lambda)$ as the mass on $\sigma$ of the ranking law with parameter $\theta$ and added masses $\lambda$.

   Note that $p_\sigma(\theta,\lambda)$ is a polynomial in the masses $\lambda$ normalised by a non-zero function in a neighbourhood of $\lambda =0$. Thus, $p_\sigma(\theta,\lambda)$ is a smooth rational functon in $(\theta,\lambda)$ because all supports are fixed. Moreover $F(\mu^\lambda(\theta))$ and $F(\mu(\theta))$ are finite sums of these rational functions and are jointly smooth in a neighbourhood of $(\theta_0,0)$. Hence
   \[
   D_\theta F(\mu^\lambda(\theta_0)) \rightarrow D_\theta F(\mu(\theta_0))
   \]
   as $\lambda \rightarrow 0$ entrywise. Hence, $\Delta(\lambda) \rightarrow \Delta(0) \neq 0$ as $\lambda \rightarrow 0$.
    
    Therefore, for sufficiently small choices of $\lambda$, there is an $r \times r$ minor of $D_\theta F(\mu^\lambda(\theta_0))$ with non-zero determinant and thus $\rank D_\theta F(\mu^\lambda(\theta_0)) \geq r$.
\end{proof}

\begin{theorem} \label{thm:induction}
     For every $n \geq 1$, there exists an open set $\Theta \subset \R^{d(n)}$, fixed pairwise disjoint finite sets $S_1,\ldots,S_n \subset \R$, and probability measures
    \[
    \mu_i(\theta) = \sum_{x \in S_i} w_{i,x}(\theta) \delta_x, \qquad 1\leq i \leq n, \quad \theta \in \Theta,
    \]
    whose positive atom masses $w_{i,x}(\theta)$ depend smoothly on $\theta$, such that, for some $\theta_0 \in \Theta$,
    \[
    \rank D_\theta F((\mu_1(\theta_0),\ldots,\mu_n(\theta_0))) \geq d(n).
    \]
\end{theorem}

\begin{proof}
We proceed by induction on $n$.

For $n=1$,  choose any single Dirac mass $\mu_1$. Since $d(1) = 0$ the rank condition holds trivially.

Assume the claim for $n-1$. Write $\mu(\theta) = (\mu_1(\theta), \ldots, \mu_{n-1}(\theta))$ and let $\theta_0$ be the parameter in the induction hypothesis. Apply Lemma \ref{lem:making_good} to add finitely many atoms to each measure in $\mu(\theta_0)$ and denote the added masses as $\lambda=(\lambda_1,\ldots,\lambda_M)$ where $M \in \N$. 
By Lemma \ref{lem:making_good}, $\mu^\lambda(\theta_0)$ is good for arbitrarily small $\lambda$. Let $\phi_{n-1}(\theta) = F(\mu^\lambda(\theta))$. By Lemma \ref{lem:staying_good} and the induction hypothesis, we may choose some sufficiently small $\lambda$, so that
\begin{equation} \label{eq:ind_hyp_rank_condition}
\rank D_\theta(\phi_{n-1}(\theta_0)) \geq d(n-1).
\end{equation}

Hence, there exist $d(n-1)$ distinct permutations $\sigma_1,\ldots,\sigma_{d(n-1)} \in S_{n-1}$ with the following property. Define the coordinate projections 
\[
r_j : \Delta(S_{n-1}) \rightarrow \R, \qquad r_j(p) = p_{\sigma_j}.
\]
The map $R(\theta)=(r_1(\phi_{n-1}(\theta)),\ldots,r_{d(n-1)}(\phi_{n-1}(\theta)))$ yields
\begin{equation} \label{eq:first_n-1_rank}
\rank D_\theta R(\theta_0) = d(n-1).
\end{equation}
Now apply Lemma \ref{lem:inserting_Xn} with the measures $\mu^\lambda(\theta_0)$ to get the measure $\mu_n(\beta) = \sum_{j=0}^{D(n)} \beta_j \delta_{y_j}$ with mass $\beta_j$ at $y_j$. Note that the points $(y_j)_{j=0}^{D(n)}$ are disjoint from the supports of each measure in $\mu^\lambda$. Let \[\mu^\lambda(\theta,\beta) =(\mu^\lambda_1(\theta),\ldots,\mu^\lambda_{n-1}(\theta),\mu_n(\beta)) 
\]
for $(\theta,\beta)\in \Theta \times V$.
Let $M:\Delta(S_n) \rightarrow \Delta(S_{n-1})$ be the marginal map obtained by projecting onto the first $n-1$ labels of each permutation. Let $\phi_n(\theta,\beta) =  F(\mu^\lambda(\theta,\beta))$. 

Define the map
\[
G(\theta,\beta) = \bra{ \bra{r_j (M \phi_n(\theta,\beta))}_{1 \leq j \leq d(n-1)} , \bra{q_{un}(\phi_n(\theta,\beta))}_{u \in W(n-1)}}.
\]
Note that  $M \phi_n(\theta,\beta) = \phi_{n-1}(\theta)$. In particular the coordinates $r_j(M \phi_n(\theta,\beta))$ do not depend on $\beta$.
Let $\beta^0 \in V$. At $(\theta_0,\beta^0)$ the Jacobian of $G$ has block form
\[
D_{\theta,\beta} G(\theta_0,\beta^0) = \begin{pmatrix}
A & 0\\
B & C
\end{pmatrix},
\]
where $A=D_\theta R(\theta_0)$ and 
\[
C = \bra{\frac{\partial}{\partial \beta_j} q_{un} (\phi_n(\theta_0,\beta^0))}_{u \in W(n-1), 1\leq j \leq D(n)}. 
\]
By \eqref{eq:first_n-1_rank}, $\rank A = d(n-1)$ and by Lemma \ref{lem:inserting_Xn}, $\rank C = D$. Hence,

\[
\rank D_{\theta,\beta} G(\theta_0,\beta^0) = d(n-1)+D = d(n).
\]
Each component of $G$ is a linear function of the coordinates of $\phi_n(\theta,\beta)$. Therefore,
\[
\rank D_{\theta,\beta} \phi_n(\theta_0,\beta^0) \geq d(n).
\]
Each measure in $(\mu^\lambda_1(\theta),\ldots,\mu^\lambda_{n-1}(\theta),\mu_n(\beta))$ has fixed finite pairwise disjoint supports, and their atom masses depend smoothly on $(\theta,\beta) \in \Theta \times V \subset \R^{d(n)}$.
\end{proof}

\begin{proof} [Proof of Theorem \ref{thm:dimension} lower bound]
    By Theorem \ref{thm:induction}, there is an open set $\Theta$ and a smooth family $\mu(\theta)$ of measures, with $\theta \in \Theta$  , such that $\rank D_\theta F(\mu(\theta_0)) \geq d(n)$ for some $\theta_0 \in \Theta$. Since $F(\mu(\Theta)) \subset \AA_n$, Proposition \ref{prop:inverse_func_thm} gives 
    \[
    \dim \AA_n \geq d(n).
    \]
\end{proof}

\begin{remark} \label{rem:support_bound}
The construction used in the proof of Theorem \ref{thm:dimension} gives an explicit bound on the sizes of the supports. Namely, the measures may be chosen so that
\begin{equation} \label{eq:support_bound}
|\supp(\mu_i)|\leq D(n)+1
=1+\sum_{j=1}^{n-1}\binom{n-1}{j}j!
\end{equation}
for every $1\leq i\leq n$. Indeed, in the induction step from $n$ to $n+1$, Lemma \ref{lem:making_good} adds $D(n+1)-D(n)$ atoms to each existing measure, while the new measure constructed in Lemma \ref{lem:inserting_Xn} has $D(n+1)+1$ atoms. Combining these two observations satsifies \eqref{eq:support_bound}. For $n\geq 2$, \eqref{eq:support_bound} is equal to $\lfloor e(n-1)!\rfloor$.
\end{remark}

\section{Mallows$_n(q)$ cannot be achieved for $n \geq 4$}

\begin{proof} [Proof of Theorem \ref{thm:mallows_unachievable}]
    We first show that $\text{Mallows}_4(q)$ is not achievable unless $q\in\{0,1\}$. Define 
    \[
     A:= \{X_1<X_3\} \qquad B:= \{X_2<X_4\}
    \]
    Since $X_1,X_2,X_3,X_4$ are independent, we have
    \[
    \P(A \cap B)= \P(A)\P(B).
    \]
    We first note that
    \[
    Z_4(q) = 1+3q+5q^2+6q^3+5q^4+3q^5+q^6 = (q+1)^2(q^2+1)(q^2+q+1)
    \]

    Suppose for a contradiction there exists a choice of distributions which makes Mallows$_4(q)$ achievable.
    Then
    \begin{align*}
        \P_q(A) &= \P_q(\sigma \in \{ 1324,1342,1234,1432,2134,4132,1243,1423,4123,2143,2413,4213\}) \\
        &= \frac{2q+1}{(q+1)(q^2+q+1)}
    \end{align*}
    By symmetry we have that $\P_q(A)=\P_q(B)$.
    Further,
    \begin{align*}
        \P_q(A \cap B) &= \P_q( \sigma \in \{1234,1324,1243,2134,2143, 2413\}) \\
        &= \frac{1+3q+q^2+q^3}{(q+1)^2(q^2+1)(q^2+q+1)}.
    \end{align*}
    Therefore,
    \begin{align*}
        \P_q(A \cap B) - \P_q(A)\P_q(B) &= \frac{1+3q+q^2+q^3}{(q+1)^2(q^2+1)(q^2+q+1)} - \frac{(2q+1)^2}{(q+1)^2(q^2+q+1)^2} \\
        &= \frac{q^3(q-1)^2}{(q+1)^2(q^2+1)(q^2+q+1)^2}.
    \end{align*}
    This can only equal $0$ when $q\in\{0,1\}$. Thus for $q>0, q\neq 1$ we have a contradiction.

    Now suppose $\text{Mallows}_n(q)$ was achievable. The relative order of the first $4$ indices must be achievable and also Mallows$_4(q)$. This contradicts the above.
\end{proof}

\section{Open Problems}
Let $\AA_{n,k}$ be the set of achievable laws realised by finitely supported measures whose supports have size at most $k$.
\begin{enumerate}
    \item The authors of \cite{babson} show that every law in $\AA_n$ can be realised by finitely supported measures, see Theorem $5.4$ \cite{babson}. What is the smallest $k=k(n)$ such that $\AA_n = \AA_{n,k}$? What is the asymptotic growth of $k(n)$?

    \item  What is $\dim \AA_{n,k}$ for each $k$? Remark $\ref{rem:support_bound}$ gives that $\dim \AA_{n,k} = \dim \AA_n$ for $k \geq \lfloor e(n-1)! \rfloor$. Can this upper bound on the least such $k$ be improved?
    
    \item Given any law $p \in \AA_n$, can one give an explicit algorithm that constructs finitely supported measures $\mu_1,\ldots,\mu_n$ whose ranking law is $p$? What is the computational complexity of such an algorithm? The construction in this paper achieves this only for a certain full dimensional subset of $\AA_n$.
    

\end{enumerate}

\appendix

\section{Semi-algebraic dimension and Jacobian rank criterion} \label{app:inverse_func_thm}

Let $A$ be a finite-dimensional real affine space, such that $A \cong \R^N$. A subset $S \subset A$ is called \textbf{semialgebraic} if, after choosing affine coordinates on $A$, it can be described by finitely many polynomial equalities and inequalities, using finitely many unions, intersections and complements.

A \textbf{$k$-dimensional smooth stratum} in $A$ is a subset $M \subset A$ with the following property. For every $x \in M$ there is an open neighbourhood $O \subset A$ of $x$ and a smooth bijection $\phi: O \rightarrow \phi(O) \subset \R^N$ with smooth inverse such that
\[
\phi(M \cap O) = \phi(O) \cap (\R^k \times \{0\}^{N-k}).
\]
In other words, near every point, $M$ can be made, by a smooth change of coordinates, to look like an open set of $\R^k \times \{0\}^{N-k}$. We write $\dim M =k$ to denote the dimension of such a stratum.

For every semialgebraic set $S$ there is a finite partition of $S$ into smooth strata $(S_\alpha)_{\alpha \in I}$. The dimension of $S$ is the maximum of the dimensions of the strata, i.e.
\[
\dim S = \max_{\alpha \in I}\dim S_\alpha.
\]
This number is independent of the choice of partition. 

\begin{proposition} \label{prop:inverse_func_thm}
    Let $U \subset \R^d$ be open, $S \subset A$ be semialgebraic and let $F: U \rightarrow A$ be a smooth map such that $F(U) \subset S$. Suppose there exists $\theta_0 \in U$ such that
    \[
        \rank D F(\theta_0) \geq d.
    \]
    Then $\dim S \geq d$.
\end{proposition}

\begin{proof}
    Choose affine coordinates on $A$ so that $A \cong \R^N$. Since $\text{rank} D F(\theta_0)\geq d$ there exists a linear projection $\pi : \R^N \rightarrow \R^d$ such that $D(\pi \circ F):\R^d \rightarrow \R^d$ is invertible. By the inverse function theorem, there is an open neighbourhood $V \subset U$ of $\theta_0$ such that $\pi \circ F(V)$ is open in $\R^d$. Since $F(V) \subset S$, $(\pi \circ F)(V) \subset \pi(S)$. Hence $\pi(S)$ contains an open set. Since $\pi$ is linear, $\pi(S)$ is semialgebraic and contains a nonempty open subset of $\R^d$. Thus $\dim \pi(S) \geq d$. Moreover, semialgebraic dimension is non-increasing under semialgebraic maps, in particular under linear projections. Hence,
    \[
    d \leq \dim \pi(S) \leq \dim S.
    \]
\end{proof}

More details can be found in Chapter $2$ and Chapter $9$ in Bochnak, Coste and Roy \cite{bochnak_coste_roy_real_algebraic_geometry}.

\section{Proof of Theorem \ref{thm:radford}} \label{app:radford}
\begin{proof} [Proof of Theorem \ref{thm:radford}]
    Clearly $w \in \ell_1 \circ \ell_2 \circ \cdots \circ \ell_s$. Let $v \in \ell_1 \circ \ell_2 \circ \cdots \circ \ell_s$ with $v \neq w$. Let $r$ be the first position $v$ and $w$ differ. There is some index $h$ so that the $r$th letter of $w$ lies in $\ell_h$. Let $s$ be the suffix of $\ell_h$ starting with the $r$th letter of $w$.

    At position $r$, $v$ cannot take a letter from any earlier factor $\ell_1,\ldots,\ell_{h-1}$ since they agreed on them. Moreover at position $r$, $v$ cannot take a letter from $\ell_h$ since the order of $\ell_h$ must be preserved. Thus there is some $j>h$ such that the $r$th letter of $v$ is in $\ell_j$. In fact it must be the first letter of $\ell_j$ since otherwise $\ell_j$ would be out of order in $v$.

    Thus it suffices to show that the first letter of $s$ is larger than the first letter of $\ell_j$.

    If $s=\ell_h$ then $s \geq \ell_j$ by Theorem \ref{thm:lyndon_factorisation}. Since $\ell_h$ and $\ell_j$ contain different letters they cannot start with the same letter and so the $r$th letter of $w$ is strictly larger than the $r$th letter of $v$.

    If instead $s$ is a proper suffix of $\ell_h$ then from the defining property of Lyndon words we have $s>\ell_h\geq \ell_j$ and we reach the same conclusion.
\end{proof}

\bibliographystyle{plain} 
\bibliography{refs.bib} 

\end{document}